\documentclass[11pt]{article}
\usepackage{indentfirst}

\usepackage{float}
\usepackage{booktabs}
\usepackage{makecell}
\usepackage{epsfig}
\usepackage{url}
\usepackage{ulem}
\usepackage{amsmath}
\usepackage{amsfonts}
\usepackage{color}
\usepackage{amssymb}
\usepackage{graphicx,amsfonts,mathrsfs}
\usepackage{epstopdf}
\usepackage{enumerate}
\usepackage[all]{xy}
\usepackage{amsthm}
\usepackage{array}

\def\sm{\rm \setminus}  
\def\A{{\rm \mathcal{A}}}

\allowdisplaybreaks

\begin{document}

\begin{sloppypar}

\newtheorem{theorem}{Theorem}[section]
\newtheorem{problem}[theorem]{Problem}
\newtheorem{corollary}[theorem]{Corollary}
\newtheorem{definition}[theorem]{Definition}
\newtheorem{conjecture}[theorem]{Conjecture}
\newtheorem{question}[theorem]{Question}
\newtheorem{lemma}[theorem]{Lemma}
\newtheorem{proposition}[theorem]{Proposition}
\newtheorem{quest}[theorem]{Question}
\newtheorem{example}[theorem]{Example}
\newcommand{\remark}{\medskip\par\noindent {\bf Remark.~~}}
\newcommand{\pp}{{\it p.}}
\newcommand{\de}{\em}

\title{  {The group vertex magicness of  unicyclic graphs and bicyclic graphs}\thanks{
 E-mail addresses: liaoqianfen@163.com(Q. Liao),  wjliu6210@126.com(W. Liu).}}

\author{Qianfen Liao, Weijun Liu$\dagger$\\
{\small School of Mathematics and Statistics, Central South University} \\
{\small New Campus, Changsha, Hunan, 410083, P.R. China. }\\
}

\maketitle

\vspace{-0.5cm}

\begin{abstract}
Let $G$ be a simple undirected graph and $\A$ an additive abelian group with identity $0$.
A mapping $\ell: V(G)\rightarrow \A\sm\{0\}$ is said to be a $\A$-vertex magic labeling of $G$ if there exists an element $\mu$ of $\A$ such that $\omega(v)=\sum_{u\in N(v)}\ell(u)=\mu$ for any vertex $v$ of $G$.
A graph $G$ that admits such a labeling is called an $\A$-vertex magic graph and $\mu$ is called magic constant.
If $G$ is $\A$-vertex magic graph for any nontrivial abelian group $\A$, then $G$ is called a group vertex magic graph.
In this paper, we give a characterization of unicyclic graphs with diameter at most $4$ which are $\A$-vertex magic.
Moreover, let $G$ be a bicyclic graph of diameter $3$, then $G$ is group vertex magic if and only if $G=M_{11}(0,0)$.
\end{abstract}

\section{Introduction}
Let $G=(V,E)$ be a simple undirected graph and $\A$ an additive abelian group with identity $0$.
For any vertex $v\in G$, $N(v)=\{u\in V: uv\in E\}$ is the set of vertices that adjacent to $u$, and $d(v)=|N(v)|$ is called the degree of $v$.

Lee et al. \cite{S2} introduced the concept of group-magic graphs as below.
\begin{definition}
Let $\A$ be an abelian group.
A graph $G=(V,E)$ is said to be $\A$-magic if there exists a labeling $\ell: E\rightarrow \A\setminus\{0\}$ such that the induced vertex labeling $\ell^+: V\rightarrow \A$ defined  by $\ell^+(v)=\sum_{uv\in E}\ell(uv)$ is a constant map.
\end{definition}

More results about group-magic graphs, we refer to \cite{S1,S2,R,W2}.
Based on the concept of group-magic graph, Kamatchi et al. \cite{K1} introduced the concept of group vertex magic graphs.
\begin{definition}
A mapping $\ell: V(G)\rightarrow \A\sm\{0\}$ is said to be an $\A$-vertex magic labeling of $G$ if there exists an element $\mu$ of $\A$ such that $\omega(v)=\sum_{u\in N(v)}\ell(u)=\mu$ for any vertex $v$ of $G$.
A graph $G$ that admits such a labeling is called an $\A$-vertex magic graph and $\mu$ is said a magic constant.
If $G$ is $\A$-vertex magic graph for any nontrivial abelian group $\A$, then $G$ is called a group vertex magic graph.
\end{definition}

After proposing the concept of group vertex magic graphs, Kamatchi et al. give a characterization of $V_4$-vertex magic trees with diameter at most $4$.
Subsequently, K.M. Sabeel et al. \cite{K2} consider the $V_4$-vertex magicness of trees with diameter $5$.
More generally, K.M. Sabeel and K. Paramasivam \cite{W1} characterize $\A$-vertex magic trees of diameter at most $5$ for any finite abelian group $\A$.
Moreover, S. Balamoorthy et al. \cite{B} discussed the products of group vertex magic graphs.

The graph $G$ is unicyclic if it contains only one cycle and $G$ is bicyclic if it contains exactly two cycles.
In this paper, we give a characterization of unicyclic graphs with diameter at most $4$ which are $\A$-vertex magic.
Moreover, we consider the group vertex magicness of bicyclic graphs with diameter $3$.

Before we get into the discussion,  some useful notations are listed here.
The distance $d(u,v)$ between vertices $u$ and $v$ is the length of a shortest path from $u$ to $v$.
The diameter of $G$ is defined by $diam(G)=\max\{d(u,v):u,v\in V\}$.
A vertex $v$ with $d(v)=1$ is a pendant vertex and the unique vertex adjacent to $v$ is a support vertex.
A vertex $v$ is called odd or even determined by the parity of its degree.
A vertex $v$ is called a weak support vertex if there is a unique pendant vertex adjacent to $v$.
If vertex $v$ is adjacent to two or more pendant vertices, then $v$ is a strong support vertex.
An element $g\in \A$ is called a square if $g\neq0$ and there exists $h\in \A$ such that $g=2h$.
Obviously, $h\neq0$.
Given a group $\A$, $e(\A)$ is the least positive integer $k$ such that $kg=0$ for any $g\in\A$.

\section{Preliminaries}

Clearly, if $G$ is $\A$-vertex magic graph with magic constant $\mu$ under mapping $\ell$ , then for any support vertex $v$, we have $\ell(v)=\mu$.
It follows that if  $G$ is $\A$-vertex magic graph and it contains support vertex, then $\mu\neq 0$.

For convenience, we propose the definition of generalized sun graph.
\begin{definition}
Let $G=(V,E)$ be a unicyclic graph with unique cycle $C_k(k<n)$.
Then $G$ is called a generalized sun, if all vertices in $V\setminus C_k$ are pendant.
\end{definition}

There is a simple but important result will be used.
\begin{proposition}\label{pro1}
Any regular graph $G$ is group vertex magic.
\end{proposition}
\begin{proof}
For any abelian group $\A$, let $g$ be a non-identity element and define $\ell(v)=g$ for any $v\in V$.
Then $G$ is $\A$-vertex magic under mapping $\ell$.
\end{proof}

To prove the main results of this paper, we need several lemmas.
\begin{lemma}\cite{K1}\label{lem0}
If $G$ is a graph with two vertices $u$ and $v$ such that $|N(u)\cap N(v)|=deg(u)-1=deg(v)$, then $G$ is not $\A$-vertex magic for any abelian group $\A$.
\end{lemma}

\begin{lemma}\label{lem1}
Let $\A$ be a  abelian group with $|\A|\geq 3$ and $g\in \A$. Then, for each $n\geq 2$, there exist $g_1, g_2, \ldots, g_n \in \A \sm \{0\}$ such that $g=g_1+g_2+\ldots+g_n$.
Hence, for any $g\in \A\setminus\{0\}$ and $n\geq 1$, there exist $g_1, g_2, \ldots, g_n \in \A \sm \{0\}$ such that $g=g_1+g_2+\ldots+g_n$.
\end{lemma}
\begin{proof}
We complete the proof by induction.
The result is true for $n=2$, since for any $g\in \A$, $g=g_1+(g-g_1)$ holds for any $g_1\in \A\setminus\{0\}$.
Now suppose the result is true for $n-1$.
Then $g=g_1+g_2+\cdots+g_{n-1}$ for $g_1, g_2,\ldots, g_{n-1}\in \A\setminus\{0\}$.
As there exist $g_{n-1}',g_n'\in \A\setminus\{0\}$ such that  $g_{n-1}=g_{n-1}'+g_n'$, we have
$g=g_1+g_2+\cdots+g_{n-1}'+g_n'$.
Thus, the result follows.
\end{proof}

It is obvious that a graph $G$ is $\mathbb{Z}_2$-vertex magic graph if and only if the degree of all vertices of $G$ are of the same parity.

\begin{lemma}\label{lem2}
Let $\A$ be an abelian group with $|\A|\geq 3$.
Then the  generalized sun graph $G$ is $\A$-vertex magic if and only if each non-pendant vertex is support vertex.
Further, if each non-pendant vertex is odd support vertex, then $G$ is group vertex magic.
\end{lemma}
\begin{proof}
Let $C_k=v_1v_2\ldots v_k$ be the unique cycle contained in $G$.
Assume $G$ is $\A$-vertex magic with magic constant $g$ under mapping $\ell$,
and there exists vertex in $C_k$ is not support vertex.
Since $G$ contains support vertex, $g\neq0$.
Observe that there exist three adjacent vertices $v_j$, $v_{j+1}$ and $v_{j+2}$ satisfy one of the following conditions.

\vspace{1mm}
\textbf{Case 1:} $v_{j+1}$ is support vertex and, $v_{j}$ and $v_{j+2}$ are not.

\vspace{1mm}
Let $v$ be other adjacent vertex of $v_j$ in $C_k$ besides $v_{j+1}$.
It is possible that $v=v_{j+2}$.
Note that $\omega(v_j)=\ell(v_{j+1})+\ell(v)=g$ and $\ell(v_{j+1})=g$, which yield that $\ell(v)=0$.

\vspace{1mm}
\textbf{Case 2:} $v_{j}$ and $v_{j+1}$ are support vertices, and $v_{j+2}$ is not.

\vspace{1mm}
Similar to \textbf{Case 1}, let $v$ be other adjacent vertex of $v_{j+2}$ in $C_k$ besides $v_{j+1}$.
Then the equalities $\omega(v_{j+2})=\ell(v_{j})+\ell(v)=g$ and $\ell(v_{j})=g$ give that $\ell(v)=0$.

\vspace{1mm}
In summary, each case above leads to a contradiction.
Hence, we complete the proof of necessity.

On the contrary, assume that each $v_i\in C_k$  is support vertex.
Let $g$ be a non-identity element of $\A$ and $\ell(v_i)=g$ for $1\leq i\leq k$.
For any vertex $v_i\in C_k$, we assume its pendant vertices are $u_1^i$, $u_2^i$, \ldots, $u_{d(v_i)-2}^i$.
By Lemma \ref{lem1}, there exist elements $g_1^i, g_2^i, \ldots, g_{d(v_i)-2}^i\in \A\sm\{0\}$ such that $\sum_{j=1}^{d(v_i)-2}g_j^i=-g$ and then we let $\ell(u_j^i)=g_j^i$.
This gives that $G$ is $\A$-vertex magic with magic constant $g$ under labeling $\ell$, for any abelian group $\A$ with $|\A|\geq 3$.

In view of $G$ contains pendant vertex, $G$ is $Z_2$-vertex magic if and only if the degree of each support vertex is odd.
Therefore, we proved the result.
\end{proof}

The next lemma is the famous Cauchy's theorem in group theory.

\begin{lemma}\label{lem3}
Let $G$ be a finite group and $p$ be a prime. If $p$ divides the order of $G$, then $G$ contains an element of order $p$.
\end{lemma}

\section{The vertex magicness of unicyclic graphs}

The unique unicyclic graph of diameter $1$ is $C_3$  which is group vertex magic.
Next we consider the unicyclic graphs of diameter $2$ and $3$, respectively.

\begin{theorem}
Let $G$ be a unicyclic graph of diameter $2$.
Then $G$ is group vertex magic if and only if $G$ is $C_4$ or $C_5$.
\end{theorem}

\begin{proof}
The unicyclic graphs of diameter $2$ are $C_4$, $C_5$ and $G_1$ as shown in Figure \ref{fig1}.
By Proposition \ref{pro1} and Lemma \ref{lem2}, $C_4$ and $C_5$  are group vertex magic graphs but
$G_1$ is not.
\begin{figure}[htbp]
\centering
\includegraphics[scale=0.8]{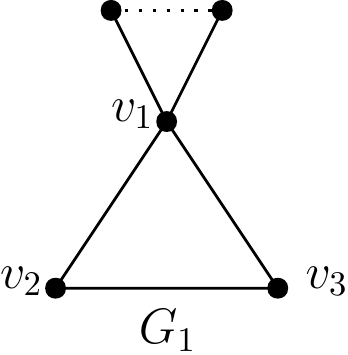}
\caption{}
\label{fig1}
\end{figure}
\end{proof}

As shown in Figure \ref{fig2}, the unicyclic graphs of diameter $3$ are divided into $4$ classes and we will discuss the group vertex magicness of each class of graphs.
\begin{figure}[htbp]
\centering
\includegraphics[scale=0.6]{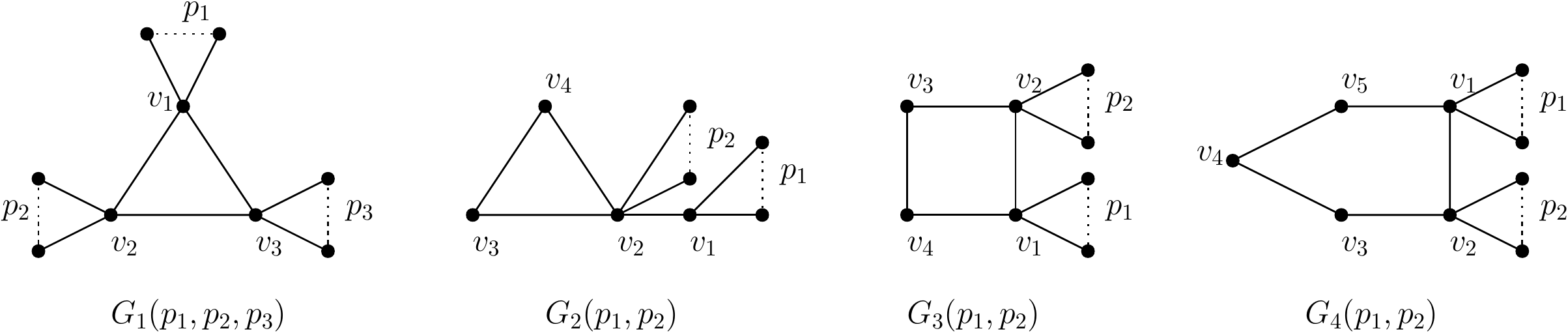}
\caption{The unicyclic graphs of diameter $3$}
\label{fig2}
\end{figure}

\begin{proposition}\label{pro3.2}
Let $\A$ be a finite  abelian group with $|\A|\geq3$.
Then $G_2(p_1,p_2)$ is $\A$-vertex magic if and only if $p_2=0$ and the order of $\A$ is even.
\end{proposition}
\begin{proof}
Suppose that $G_2(p_1,p_2)$ is $\A$-vertex magic with magic constant $g\neq 0$ under mapping $\ell$.
As the diameter of $G_2(p_1,p_2)$ is $3$, $p_1\neq0$ and then $\ell(v_1)=g$.
If $p_2\neq0$, then $\ell(v_2)=g$.
Since $\omega(v_3)=\ell(v_2)+\ell(v_4)=g$, we obtain that $\ell(v_4)=0$, which is a contradiction.
Thus, $p_2=0$.
According to
\begin{align*}
\omega(v_2)=g=\ell(v_1)+\ell(v_3)+\ell(v_4)=g+\ell(v_3)+\ell(v_4),
\end{align*}
we have $\ell(v_3)=-\ell(v_4)$.
Thus, $\ell(v_2)=\omega(v_3)-\ell(v_4)=g+\ell(v_3)$.
As a result, $\omega(v_4)=\ell(v_2)+\ell(v_3)=g+2\ell(v_3)=g$, which follows that $\ell(v_3)$ is an involution of $\A$.
Therefore, the order of $\A$ is even.

Conversely, if $p_2=0$ and the order of $\A$ is even.
Let $h$ be an involution of $\A$.
Choose $g\in \A\setminus\{0,h\}$.
Define $\ell(v_1)=g$, $\ell(v_2)=g-h$ and $\ell(v_3)=\ell(v_4)=h$.
Assume the pendant vertices adjacent to $v_1$ is $u_1$, \ldots, $u_{d(v_1)-1}$.
By Lemma \ref{lem1}, there exist $g_1, g_2, \ldots, g_{d(v_1)-1}\in \A\setminus\{0\}$ such that $\sum_{j=1}^{d(v_1)-1}g_j=h$.
Then we define $\ell(u_j)=g_j$, for $1\leq j\leq d(v_1)-1$.
It is easy to check that $\ell$ is $\A$-vertex magic labeling of $G_2(p_1, 0)$ and the magic constant is $g$.
\end{proof}

Observe that  $G_2(p_1,p_2)$, $G_3(p_1, p_2)$ and $G_4(p_1, p_2)$ are not $Z_2$-vertex magic.
By Lemma \ref{lem2}, the characterization of group vertex magic unicyclic graphs with diameter $3$ is given.
\begin{theorem}
Let $G$ be a unicyclic graph of diameter $3$.
Then $G$ is group-vertex magic graph if and only if $G=G_1(p_1, p_2, p_3)$, where $p_i\geq 1$ and $p_i$ is odd, for each $1\leq i\leq3$.
\end{theorem}

Now we consider the group vertex magic unicyclic graphs of diameter $4$ which are shown in Figure \ref{fig3}.
\begin{figure}[t]
\centering
\includegraphics[scale=0.8]{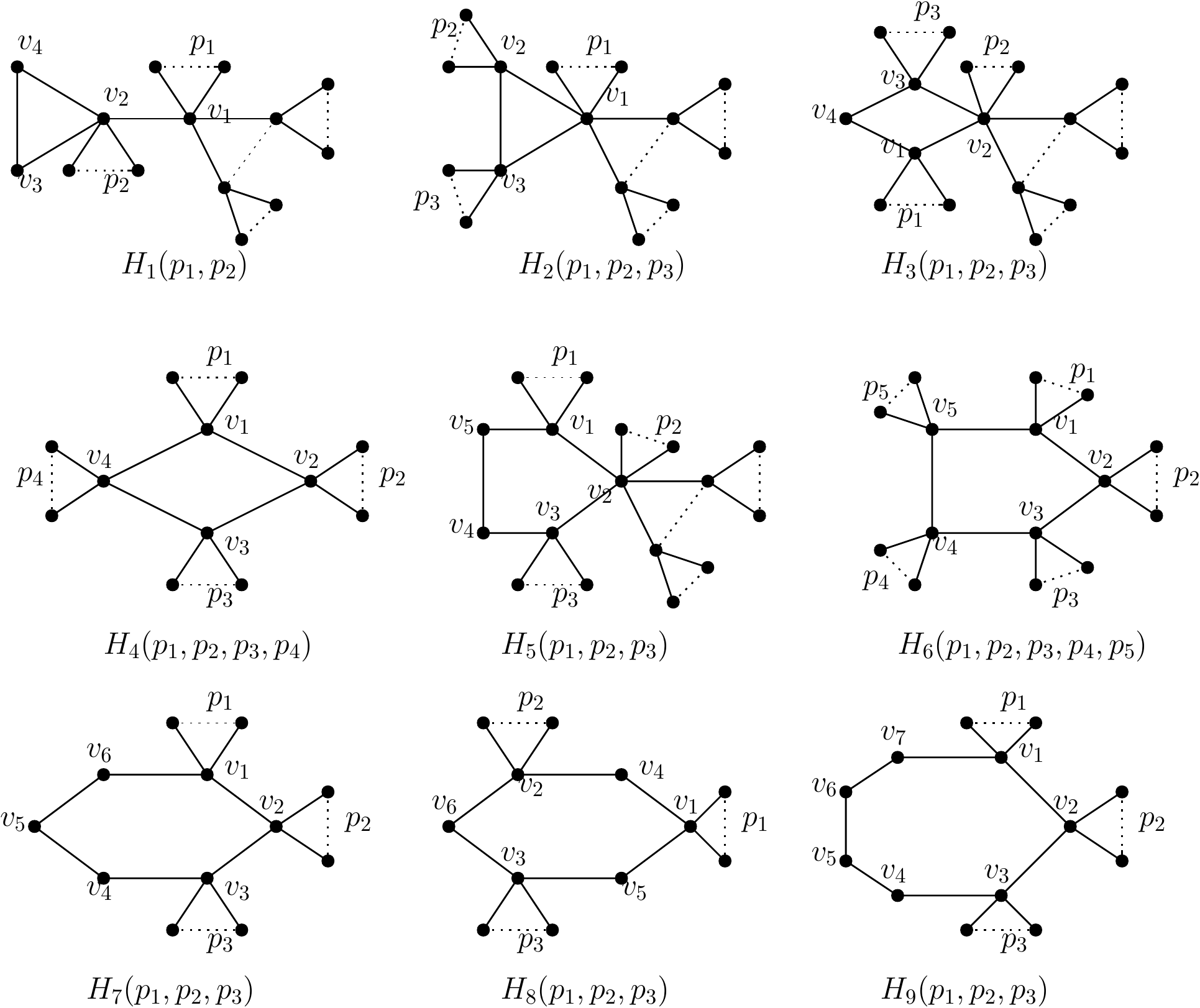}
\caption{The unicyclic graphs of diameter $4$}
\label{fig3}
\end{figure}

\begin{proposition}\label{pro3.4}
Let $\A$ be a finite abelian group with $|\A|\geq3$.
Then $H_1(p_1,p_2)$ with $p_1=0$ is  $\A$-vertex magic if and only if  $p_2=0$, and one of the following conditions holds:
\begin{enumerate}[{\rm(i)}]
\item if $v_1$ has a weak support neighbor, then there exists $g\in \A\setminus\{0\}$ such that $d(v_1)\not\equiv1 (mod\ o(g))$, $d(v_1)\not\equiv2(mod\ o(g))$, $2d(v_1)\not\equiv 2 (mod\ o(g))$ and $2d(v_1)\not\equiv 3 (mod\ o(g))$;
\item if each support neighbor of $v_1$ is strong support vertex, then there exists $g\in \A\setminus\{0\}$ such that $d(v_1)\not\equiv1 (mod\ o(g))$, $d(v_1)\not\equiv2(mod\ o(g))$ and $2d(v_1)\not\equiv 3 (mod\ o(g))$.
\end{enumerate}
\end{proposition}
\begin{proof}
Suppose that $H_1(0,p_2)$ is  $\A$-vertex magic with magic constant $g\neq 0$ under mapping $\ell$.
If $p_2\neq0$, then $\ell(v_2)=g$.
Thus, $\omega(v_3)=\ell(v_2)+\ell(v_4)=g+\ell(v_4)=g$, which yields that $\ell(v_4)=0$.
This leads to a contradiction.
Hence, $p_2=0$.
Since $\omega(v_1)=\ell(v_2)+(d(v_1)-1)g=g$, we obtain $\ell(v_2)=(2-d(v_1))g$.
Thus, $\ell(v_3)=\omega(v_4)-\ell(v_2)=(d(v_1)-1)g$ and $\ell(v_4)=\omega(v_3)-\ell(v_2)=(d(v_1)-1)g$ follows.
Further, we arrive at $\ell(v_1)=\omega(v_2)-2\ell(v_3)=(3-2d(v_1))g$.
Since $\ell$ is a mapping from $V$ to $\A\setminus\{0\}$, $(2-d(v_1))g$, $(d(v_1)-1)g$ and $(3-2d(v_1))g$ are all not $0$.
Moreover, if $v_1$ has weak support neighbor $x$ and let $y$ be the unique pendant neighbor of $x$.
Then $\ell(y)=\omega(x)-\ell(v_1)=(2d(v_1)-2)g\neq 0$ holds.

Conversely, for $H_2(0,0)$, if it satisfies (i) or (ii), we label all support vertices with non-identity element $g$ and, let $\ell(v_1)=(3-2d(v_1))g$, $\ell(v_2)=(2-d(v_1))g$ and $\ell(v_3)=\ell(v_4)=(d(v_1)-1)g$.
For the pendant neighbors of any strong support vertex, by Lemma \ref{lem1}, we can label them such that their label sum is $(2d(v_1)-2)g$.
Especially, if $v_1$ has weak support neighbor $x$, let $y$ be the unique pendant neighbor of $x$, let $\ell(y)=(2d(v_1)-2)g$.
It can be checked that  $H_2(0,0)$ is a $\A$-vertex magic graph with magic constant $g$ under labeling $\ell$.
\end{proof}

\begin{proposition}\label{pro3.5}
Let $\A$ be a finite abelian group with $|\A|\geq3$.
Then $H_1(p_1,p_2)$ with $p_1\neq0$ is  $\A$-vertex magic if and only if $p_2=0$, and one of the following condition holds:
\begin{enumerate}[{\rm(i)}]
\item if $p_1=1$, then each support neighbor of $v_1$ is strong support vertex and there exist an involution $h$ of $\A$ and $g\in\A\setminus\{0,h\}$ such that $h\neq (d(v_1)-2)g$;
\item if $p_1\geq2$, then each support neighbor of $v_1$ is strong support vertex and the order of $\A$ is even.
\end{enumerate}
\end{proposition}
\begin{proof}
Suppose $H_1(p_1,p_2)$ with $p_1\neq0$ is a $\A$-vertex magic graph with magic constant $g$ under $\ell$.
Using the same argument as in the proof of Proposition \ref{pro3.4}, we can easily obtain  $p_2=0$.
We assert that each support neighbor of $v_1$ is strong support vertex.
If the assertion not true, then there exists a weak support neighbor $x$ of $v_1$, and let $y$ be the unique pendant neighbor of $x$.
Then $\ell(y)=\omega(x)-\ell(v_1)=0$, which is a contradiction.
So, our assertion is follows.
From $\omega(v_3)=\omega(v_4)$, we obtain that $\ell(v_3)=\ell(v_4)$.
Thus, $\omega(v_2)=\ell(v_1)+\ell(v_3)+\ell(v_4)=g+2\ell(v_3)=g$, which indicates $\ell(v_3)$ is an involution of $\A$.
Assume $\ell(v_3)=h$, where $h$ is an involution of $\A$.
Then $\ell(v_2)=\omega(v_4)-\ell(v_3)=g-h\neq0$.
Especially, if $p_1=1$, let $v$ be the unique pendant neighbor of $v_1$, then we have $\ell(v)=\omega(v_1)-\ell(v_2)-(d(v_1)-2)g=(2-d(v_1))g+h\neq0$.

Conversely, suppose $H_1(p_1,0)$ and $\A$ satisfy  (i) or (ii).
Let $h$ be an involution of $\A$ and $g\in\A\setminus \{0,h\}$.
We label all support vertices with $g$ and, let $\ell(v_3)=\ell(v_4)=h$ and $\ell(v_2)=g-h$.
For any strong support neighbor $v$ of $v_1$, by Lemma \ref{lem1}, we can label the pendant neighbors of $v$ such that their label sum is $0$.
In addition, label the pendant neighbors of $v_1$ such that their label sum is $(1+p_1-d(v_1))g+h$.
Then, $H_1(p_1,0)$ is a $\A$-vertex magic graph with magic constant $g$ under the labeling $\ell$.
\end{proof}

It is clear that $H_1(p_1,p_2)$ is not $Z_2$-vertex magic.
Hence, we have
\begin{corollary}
$H_1(p_1,p_2)$ is not group vertex magic.
\end{corollary}

\begin{proposition}\label{pro3.7}
Let $\A$ be a finite abelian group with $|\A|\geq3$.
Then $H_2(p_1,p_2,p_3)$ is $\A$-vertex magic if and only if both $p_2$ and $p_3$ are non-zero and, one of the following condition holds:
\begin{enumerate}[{\rm(i)}]
\item $p_1=0$ and $gcd(d(v_1)-1,|\A|)\neq 1$;
\item $p_1=1$ and except for $v_2$ and $v_3$, all other support neighbors of $v_1$ are strong support and $d(v_1)\not\equiv 2(mod\ e(\A))$;
\item except for $v_2$ and $v_3$, all other non-pendant vertices are strong support.
\end{enumerate}
\end{proposition}
\begin{proof}
Suppose that $H_2(p_1,p_2,p_3)$ is  $\A$-vertex magic with magic constant $g\neq 0$ under mapping $\ell$.
Firstly, notice that $p_2$ and $p_3$ cannot all be $0$, since the diameter of $H_2(p_1,p_2,p_3)$ is $4$.
Without loss of generality, assume $p_2=0$ and $p_3\neq 0$, the equality $\omega(v_2)=\ell(v_1)+\ell(v_3)=\ell(v_1)+g=g$ gives that $\ell(v_1)=0$  which is a contradiction.
Hence, both $p_2$ and $p_3$ are non-zero.

If $p_1=0$, then $\omega(v_1)=d(v_1)g=g$, it follows that $(d(v_1)-1)g=0$ and so $o(g)$ is a common divisor of  $d(v_1)-1$ and $|\A|$.
Next the condition of $p_1\geq 1$ is discussed.
We point out that, except for $v_2$ and $v_3$, each support neighbor of $v_1$ is strong support.
In fact, if $x$ is a weak support neighbor of $v_1$, where $x\neq v_2$ and $x\neq v_3$,
and $y$ is the unique pendant vertex adjacent to $x$.
Then $\ell(y)=\omega(x)-\ell(v_1)=0$, which is a contradiction.
Especially, if $p_1=1$ and assume the unique pendant neighbor of $v_1$ is $z$.
Observe that $\ell(z)=\omega(v_1)-(d(v_1)-1)g=(2-d(v_1))g\neq0$, so $d(v_1)\not\equiv 2(mod\ e(\A))$.

Conversely, consider $H_2(p_1,p_2,p_3)$ with $p_2\neq 0$ and $p_3\neq 0$.
If $H_2(p_1,p_2,p_3)$ satisfies (i), we let $gcd(d(v_1)-1,|\A|)=m\neq1$ and $p$ be a prime divisor of $m$.
By Cauchy's theorem, $\A$ has an element $g$ of order $p$.
Hence, $o(g)$ divides $d(v_1)-1$ and then $d(v_1)g=g$.
Choose element $h\in \A\setminus\{0,g\}$ and we define $\ell(v)=g$ for each support vertex $v$, and $\ell(v_1)=h$.
For the pendant neighbors of $v_i$, where $i=2$ or $3$, by Lemma \ref{lem1}, we can label them  such that their label sum is $-h$.
In addition, for any support neighbor $u$ of $v_1$ except for $v_2$ and $v_3$, we label the pendant neighbors of $u$ such that their label sum is $g-h$.
This gives an $\A$-vertex labeling of  $H_2(0,p_2,p_3)$.

If $H_2(p_1,p_2,p_3)$ satisfies (ii), there exists element $g\in\A\setminus\{0\}$ such that $(2-d(v_1))g\neq 0$.
Let $u$ be the unique pendant neighbor of $v_1$.
Define $\ell(u)=(2-d(v_1))g$ and $\ell(v)=g$, for each support vertex $v$.
For pendant neighbors of $v_i$, where $i=2$ or $3$, by Lemma \ref{lem1} again, we can label them such that their sum is $-g$.
As regard to the pendant neighbors of other support vertex $v$, except for $v_2$ and $v_3$, we label them such that their sum is $0$.
Then $H_2(1,p_2,p_3)$ is $\A$-vertex magic graph under mapping $\ell$.

If $H_2(p_1,p_2,p_3)$ satisfies (iii), let $g\in \A\setminus\{0\}$ and we label all support vertices with $g$.
Except for $v_1$, $v_2$ and $v_3$, for any other support vertex $v$, we label the pendant neighbors of $v$ such that their label sum is $0$.
For $i=2$ or $3$, label the pendant neighbors of $v_i$ such that their label sum is $-g$.
At last, label the pendant neighbors of $v_1$ such that their label sum is $(p_1+1-d(v_1))g$.
Clearly, $\ell$ is an $\A$-vertex magic labeling of $H_2(p_1,p_2,p_3)$ with magic constant $g$.
\end{proof}

\begin{corollary}\label{cor3.8}
$H_2(p_1,p_2,p_3)$ is group vertex magic if and only if  $v_2$ and $v_3$ are odd support and, except for $v_2$ and $v_3$, all  other non-pendant vertiecs are odd strong support.
\end{corollary}
\begin{proof}
$H_2(p_1,p_2,p_3)$ is $Z_2$-vertex magic implies all non-pendant vertices are odd.
Let $\A$ be the cyclic group of order $d(v_1)$, then $gcd(d(v_1)-1,|\A|)=1$.
If $p_1=1$, then $d(v_1)\geq 4$.
Let $\A$ be the cyclic group of order $d(v_1)-2$, then $d(v_1)\equiv 2(mod\ e(\A))$.
Thus, by Proposition \ref{pro3.7}, $H_2(p_1,p_2,p_3)$ is group vertex magic if and only if  $v_2$ and $v_3$ are odd support and, except for $v_2$ and $v_3$, all  other non-pendant vertiecs are odd strong support.
\end{proof}

\begin{proposition}\label{pro3.9}
Let $\A$ be a finite abelian group with $|\A|\geq3$.
Then $H_3(p_1,p_2,p_3)$ is $\A$-vertex magic if and only if $p_1=p_2=p_3=0$ and $gcd(d(v_2)-2,|\A|)\neq 1$.
\end{proposition}
\begin{proof}
Suppose that $H_3(p_1,p_2,p_3)$ is $\A$-vertex magic with magic constant $g\neq0$ under mapping $\ell$.
If one of $p_1$ and $p_3$ is not $0$, without loss of generality, assume $p_1\neq 0$.
Then $\ell(v_3)=\omega(v_4)-\ell(v_1)=g-g=0$, which leads to a contradiction.
Hence, $p_1=p_3=0$.
If $p_2\neq0$, then $\ell(v_2)=g$.
It follows that $\ell(v_4)=\omega(v_3)-\ell(v_2)=0$, which is also a contradiction.
Thus, $p_2=0$.
Since $\omega(v_4)=\ell(v_1)+\ell(v_3)=g$, it follows that
\begin{align*}
\omega(v_2)=\ell(v_1)+\ell(v_3)+(d(v_2)-2)g=(d(v_2)-1)g=g.
\end{align*}
Hence, $(d(v_2)-2)g=0$, which implies $o(g)$ is a common divisor of $d(v_2)-2$ and $|\A|$.

Conversely, let $m=gcd(d(v_2)-2,|\A|)>1$ and $p$ be a prime divisor of $m$.
By Cauchy's theorem, $\A$ has an element $g$ of order $p$.
By Lemma \ref{lem1}, there exists $g_1, g_2\in \A\setminus\{0\}$ such that $g=g_1+g_2$.
For $H_3(0,0,0)$, define $\ell(v_1)=\ell(v_2)=g_1$, $\ell(v_3)=\ell(v_4)=g_2$.
For each support vertex $v$ in graph, define $\ell(v)=g$.
In addition, for the pendant neighbors of $v$, we label them such that their label sum is $g_2$.
Now, $H_3(0,0,0)$ is $\A$-vertex magic graph with magic constant $g$ under $\ell$.
\end{proof}

\begin{proposition}\label{pro3.10}
Let $\A$ be a finite abelian group with $|\A|\geq3$.
Then $H_5(p_1,p_2,p_3)$ is $\A$-vertex magic if and only if $p_1=p_2=p_3=0$ and there exist two different vertices $g,h\in \A\setminus\{0\}$ such that $2h=(3-d(v_2))g$.
\end{proposition}
\begin{proof}
Suppose $H_5(p_1,p_2,p_3)$ is $\A$-vertex magic graph with magic constant $g\neq 0$ under mapping $\ell$.
If $p_1\neq 0$, then $\ell(v_1)=g$ and  $\ell(v_4)=\omega(v_5)-\ell(v_1)=0$, which is a contradiction.
Using the similar argument, we obtain $p_3=0$.
Also, we claim $p_2= 0$.
Otherwise, $\ell(v_2)=g$ yields that $\ell(v_5)=\omega(v_1)-\ell(v_2)=0$, which is also a contradiction.
According to $\omega(v_1)=\omega(v_3)=\omega(v_4)=\omega(v_5)$, we have
\begin{align*}
\ell(v_2)+\ell(v_5)=\ell(v_2)+\ell(v_4)=\ell(v_3)+\ell(v_5)=\ell(v_1)+\ell(v_4)=g,
\end{align*}
then the equalities $\ell(v_1)=\ell(v_2)=\ell(v_3)$ and $\ell(v_4)=\ell(v_5)$ hold.
Assume $\ell(v_2)=h$, then $\ell(v_4)=\ell(v_5)=g-h\neq0$ and $\ell(v_3)=h=\omega(v_2)-\ell(v_1)-(d(v_2)-2)g=(3-d(v_2))g-h$.
Thus, $(3-d(v_2))g=2h$.

Conversely, there exist different elements $g, h\in \A\setminus\{0\}$ such that $2h=(3-d(v_2))g$.
For $H_5(0,0,0)$, define $\ell(v_1)=\ell(v_2)=\ell(v_3)=h$ and $\ell(v_4)=\ell(v_5)=g-h$.
For each support neighbor $u_j$ of $v_2$, define $\ell(u_j)=g$.
In addition, for pendant neighbors of $u_j$, by Lemma \ref{lem1}, we can label them such that their label sum is $g-h$.
Then, $H_5(0,0,0)$ is a $\A$-vertex magic graph under labeling $\ell$.
\end{proof}

\begin{corollary}
Let $\A$ be a finite abelian group with $|\A|\geq3$.
Then $H_5(0,0,0)$ with $d(v_2)=3$ is $\A$-vertex magic graph if and only if the order of $\A$ is even.
\end{corollary}

Observe that $H_4(p_1,p_2,p_3)$, $H_6(p_1,p_2,p_3,p_4,p_5)$, $H_7(p_1,p_2,p_3)$, $H_8(p_1,p_2,p_3)$ and $H_9(p_1,p_2,p_3)$ are all generalized sun graphs.
By Lemma \ref{lem2}, $H_7(p_1,p_2,p_3)$, $H_8(p_1,p_2,p_3)$ and $H_9(p_1,p_2,p_3)$ cannot be $\A$-vertex magic for any abelian group $\A$.

At the end of this section, we determine the group vertex magic unicyclic graphs with diameter $4$.
\begin{theorem}
Let $G$ be a group vertex magic unicyclic graphs with diameter $4$.
Then is one of the following condition holds:
\begin{enumerate}[{\rm(i)}]
\item $G=H_2(p_1,p_2,p_3)$,  where $p_2$ and $p_3$ are odd and, except for $v_2$ and $v_3$, all  other non-pendant vertices are odd strong support;

\item $G=H_4(p_1,p_2,p_3)$, where $p_i\geq1$ and $p_i$ is odd, for each $1\leq i\leq 3$;

\item $G=H_6(p_1,p_2,p_3,p_4,p_5)$, where $p_i\geq1$ and $p_i$ is odd, for each $1\leq i\leq 5$;.

\end{enumerate}
\end{theorem}

\section{The vertex magicness of bicyclic graphs}

In this section, we discuss the vertex magicness of bicyclic graphs with diameter $3$ which are  shown in Figure \ref{fig4}.
By analyzing each class of graphs respectively, we obtain the following  results.
\begin{figure}[htbp]
\centering
\includegraphics[scale=0.8]{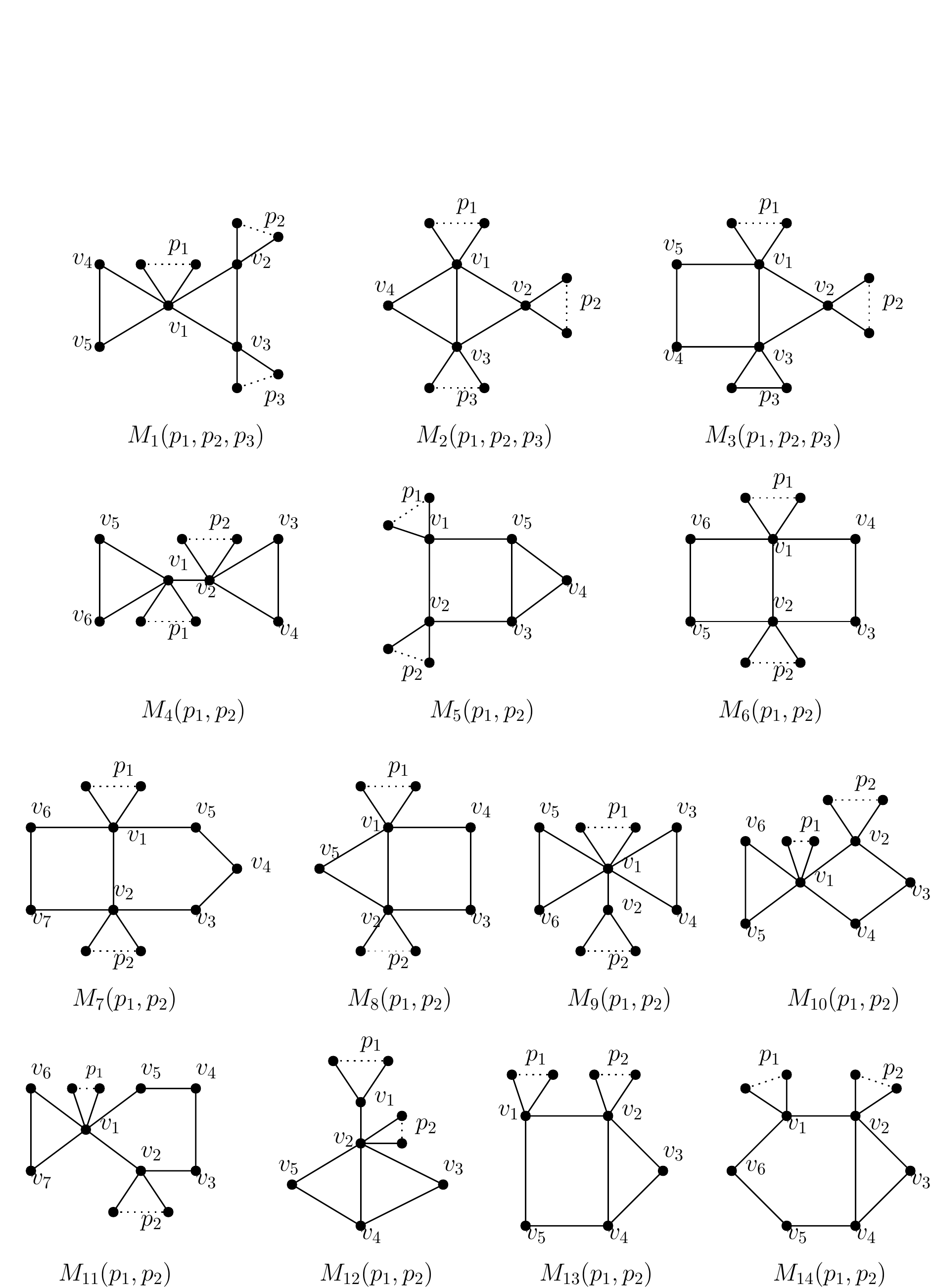}
\caption{The bicyclic graphs of diameter $3$}
\label{fig4}
\end{figure}

\begin{proposition}\label{pro4.1}
Let $\A$ be a finite abelian group with $|\A|\geq 3$.
Then $M_1(p_1,p_2,p_3)$ is $\A$-vertex magic if and only if $p_1=0$, $p_2$ and $p_3$ are non-zero, and $\A$ contains a square element.
\end{proposition}
\begin{proof}
Suppose $M_1(p_1,p_2,p_3)$ is $\A$-vertex magic with magic constant $g$ under mapping $\ell$.
We first point out that $p_1=0$.
Otherwise, if $p_1\neq0$, then $\ell(v_1)=g$ and $\ell(v_5)=\omega(v_4)-\ell(v_1)=0$, which is a contradiction.
Observe that $p_2$ and $p_3$ cannot all be zero, since the diameter of $M_{1}(p_1,p_2,p_3)$ is $3$.
Without loss of generality, assume that $p_2=0$ and $p_3\neq0$.
Then $\ell(v_1)=\omega(v_2)-\ell(v_3)=g-g=0$.
Hence, both $p_2$ and $p_3$ are non-zero.
According to $\omega(v_4)=\omega(v_5)=g$, we obtain $\ell(v_4)=\ell(v_5)=g-\ell(v_1)$.
Thus, $\ell(v_1)=\ell(v_2)+\ell(v_3)+2\ell(v_4)=2g+2\ell(v_4)=g$, which means that $g$ is a square element of $\A$.

Conversely, assume $g$ is a square of $\A$ and $g=2h$.
For $M_1(0,p_2,p_3)$, define $\ell(v_1)=g+h$, $\ell(v_2)=\ell(v_3)=g$ and $\ell(v_4)=\ell(v_5)=-h$.
For each support vertex $v_i$, where $i=1~\text{or}~2$, we label the pendant neighbors of $v_i$
such that their label sum is $-h$.
Now, $M_1(0,p_2,p_3)$ is vertex magic with magic constant $g$ under $\ell$.
\end{proof}

\begin{proposition}\label{pro4.2}
Let $\A$ be a finite abelian group with $|\A|\geq 3$.
Then $M_2(p_1,p_2,p_3)$ is $\A$-vertex magic if and only if $p_1=p_3=0$, $p_2\geq2$, and $\A$ contains a square element.
\end{proposition}
\begin{proof}
Suppose $M_2(p_1,p_2,p_3)$ is $\A$-vertex magic with magic constant $g\neq 0$ under mapping  $\ell$.
If one of $p_1$ and $p_3$ is not zero, without loss of generality, assume $p_1\neq0$.
Then $\ell(v_3)=\omega(v_4)-\ell(v_1)=0$, which is a contradiction.
Hence, $p_1=p_3=0$.
Now, in view of  the diameter of $M_2(p_1,p_2,p_3)$ is $3$, $p_2\neq0$ and then $\ell(v_2)=g$.
If $p_2=1$, let $u$ be the unique pendant vertex of $v_2$.
Then $\ell(u)=\omega(v_2)-\ell(v_1)-\ell(v_3)=\omega(v_2)-\omega(v_4)=0$, which is also a contradiction.
Thus, $v_2$ is a strong support vertex.
According to $\omega(v_1)=\ell(v_2)+\ell(v_3)+\ell(v_4)=g$ and $\omega(v_3)=\ell(v_1)+\ell(v_2)+\ell(v_4)=g$, we have $\ell(v_1)=\ell(v_3)$.
Thus, $\omega(v_4)=g=2\ell(v_1)$, which means $g$ is a square of $\A$.

Conversely, let $g$ be a square of $\A$ and $g=2h$.
Define $\ell(v_1)=\ell(v_3)=h$, $\ell(v_4)=-h$ and $\ell(v_2)=g$.
For the pendant neighbors of $v_2$, label them such that their label sum is $0$.
Then $M_2(0,p_2,0)$ is $\A$-vertex magic with magic constant $g$ under mapping $\ell$.
\end{proof}

\begin{proposition}\label{pro4.3}
For any abelian group $\A$, $M_3(p_1,p_2,p_3)$ is not $\A$-vertex magic.
\end{proposition}
\begin{proof}
The result will be proved by negation.
Suppose that $M_3(p_1,p_2,p_3)$ is $\A$-vertex magic with constant $g$ under $\ell$.
We assert $p_1=p_3=0$.
If $p_1\neq0$, then $\ell(v_1)=g$ and $\ell(v_4)=\omega(v_5)-\ell(v_1)=0$, which leads to a contradiction.
With similar argument, we obtain $p_3=0$.
However, the equality $\omega(v_3)=\ell(v_1)+\ell(v_2)+\ell(v_4)=\omega(v_5)+\ell(v_2)$ gives that $\ell(v_2)=0$, which is also a contradiction.
Thus, $M_3(p_1,p_2,p_3)$ is not $\A$-vertex magic, for any abelian group $\A$.
\end{proof}

\begin{proposition}\label{pro4.4}
Let $\A$ be a finite abelian group with $|\A|\geq 3$.
Then $M_4(p_1,p_2)$ is $\A$-vertex magic if and only if $p_1=p_2=0$ and there exist different elements $g,h\in \A\setminus\{0\}$ such that $g\neq2h$, $2g\neq2h$ and $3g=3h$.
\end{proposition}
\begin{proof}
Suppose that $M_4(p_1,p_2)$ is $\A$-vertex magic with magic constant $g$ under mapping $\ell$.
If $p_1\neq0$, then $\ell(v_6)=\omega(v_5)-\ell(v_1)=g-g=0$, which is a contradiction.
Similarly, we can prove that $p_2=0$.
According to $\omega(v_3)=\omega(v_4)$, we have $\ell(v_3)=\ell(v_4)=g-\ell(v_2)$.
Also, from  $\omega(v_5)=\omega(v_6)=g$, $\ell(v_5)=\ell(v_6)=g-\ell(v_1)$ follows.
Thus, $\omega(v_2)=\ell(v_1)+\ell(v_3)+\ell(v_4)=\ell(v_1)+2g-2\ell(v_2)=g$, which implies $\ell(v_1)=2\ell(v_2)-g$.
Moreover, the equality $\omega(v_1)=2(g-\ell(v_1))+\ell(v_2)=4g-3\ell(v_2)=g$ yields that $3g=3\ell(v_2)$.
Based on the fact that $\ell(v_1)=2\ell(v_2)-g$, $\ell(v_3)=g-\ell(v_2)$ and $\ell(v_5)=2g-2\ell(v_2)$ are all not $0$, the result follows.

Conversely, define $\ell(v_1)=2h-g$, $\ell(v_2)=h$, $\ell(v_3)=\ell(v_4)=g-h$ and $\ell(v_5)=\ell(v_6)=2g-2h$.
Then, $\ell$ is $\A$-vertex magic labeling of $M_4(0,0)$.
\end{proof}

\begin{proposition}\label{pro4.5}
Let $\A$ be a finite abelian group with $|\A|\geq 3$.
Then $M_5(p_1,p_2)$ is $\A$-vertex magic if and only if both $p_1$ and $p_2$ are not zero and, $\A$ contains square element.
\end{proposition}
\begin{proof}
Suppose $M_5(p_1,p_2)$ is $\A$-vertex magic with magic constant $g\neq0$ under mapping $\ell$.
Since the diameter of $M_5(p_1,p_2)$ is $3$, $p_1$ and $p_2$ cannot all be zero.
Without loss of generality, assume $p_1=0$ and $p_2\neq 0$.
Then $\ell(v_5)=\omega(v_1)-\ell(v_2)=0$, which is a contradiction.
Hence, both $p_1$ and $p_2$ are not zero and then $\ell(v_1)=\ell(v_2)=g$.
According to $\omega(v_3)=\ell(v_2)+\ell(v_4)+\ell(v_5)=g$ and $\omega(v_5)=\ell(v_1)+\ell(v_3)+\ell(v_4)$, we obtain $\ell(v_3)=\ell(v_5)=-\ell(v_4)$.
Thus, $\omega(v_4)=2\ell(v_3)=g$, which means $g$ is a square of $\A$.

Conversely, there exists $g\in \A\setminus\{0\}$ such that $g=2h$ for some $h\in A\setminus\{0\}$.
Define $\ell(v_1)=\ell(v_2)=g$, $\ell(v_4)=-h$ and $\ell(v_3)=\ell(v_5)=h$.
For the support element $v_i$, where $i=1~\text{or}~2$, label the pendant neighbors of $v_i$ such their label sum is $-h$.
It is easy to check that $M_5(p_1,p_2)$ is vertex magic with magic constant $g$ under $\ell$.
\end{proof}

Due to the group $V_4$ does not contain square element, an immediate consequence of above proposition follows.
\begin{corollary}\label{cor4.6}
Graphs $M_1(p_1,p_2,p_3)$, $M_2(p_1,p_2,p_3)$, $M_4(p_1,p_2)$ and $M_5(p_1,p_2)$cannot be group vertex magic.
\end{corollary}

\begin{proposition}\label{pro4.7}
For any abelian group $\A$, $M_6(p_1,p_2)$, $M_7(p_1,p_2)$ and $M_8(p_1,p_2)$  are not $\A$-vertex magic graphs.
\end{proposition}
\begin{proof}
Suppose $M_6(p_1,p_2)$ is $\A$-vertex magic graph, then $p_1=p_2=0$.
Otherwise, without loss of generality, assume $p_1\neq 0$.
Then $\ell(v_3)=\omega(v_4)-\ell(v_1)=0$, which is a contradiction.
However, for the case of $p_1=p_2=0$, $|N(v_1)\cap N(v_3)|=deg(v_1)-1=deg(v_3)$.
By Lemma \ref{lem0}, $M_6(0,0)$ is not $\A$-vertex magic graph.

The discussion for $M_7(p_1,p_2)$ is similar to $M_6(p_1,p_2)$.
If $M_7(p_1,p_2)$ is vertex magic graph, then $p_1=p_2=0$.
Observe that in $M_7(0,0)$, $|N(v_1)\cap N(v_7)|=deg(v_1)-1=deg(v_7)$, so applying Lemma \ref{lem0} again, we have $M_7(0,0)$ is not $\A$-vertex magic.

Suppose that $M_8(p_1,p_2)$ is vertex magic graph with magic constant $g$ under $\ell$.
Since the diameter of $M_8(p_1,p_2)$ is $3$, $p_1$ and $p_2$ cannot all be zero.
Without loss of generality, assume $p_1\neq0$.
Then $\ell(v_3)=\omega(v_4)-\ell(v_1)=0$, which is a contradiction.
Hence, $M_8(p_1,p_2)$ is not $\A$-vertex magic graph.
\end{proof}

\begin{proposition}\label{pro4.8}
Let $\A$ be a finite abelian group with $|\A|\geq 3$.
Then $M_9(p_1,p_2)$ is $\A$-vertex magic if and only if $p_1=0$ and there exist different elements $g,h\in \A\setminus\{0\}$ such that $4(g-h)=0$.
\end{proposition}
\begin{proof}
Suppose $M_9(p_1,p_2)$ is $\A$-vertex magic with magic constant $g$ under mapping $\ell$.
If $p_1\neq 0$, then $\ell(v_6)=\omega(v_5)-\ell(v_1)=0$, which is a contradiction.
Hence, $p_1=0$.
According to $\omega(v_3)=\omega(v_4)$, we have  $\ell(v_3)=\ell(v_4)=g-\ell(v_1)$.
Similarly, from $\omega(v_5)=\omega(v_6)$, $\ell(v_5)=\ell(v_6)=g-\ell(v_1)$ follows.
Hence, $\ell(v_3)+\ell(v_4)+\ell(v_5)+\ell(v_6)=4(g-\ell(v_1))=\omega(v_1)-\ell(v_2)=0$.

Conversely, if $p_1=0$ and there exist different element $g,h\in \A\setminus\{0\}$ such that $4(g-h)=0$.
Define $\ell(v_1)=h$, $\ell(v_2)=g$ and $\ell(v_3)=\ell(v_4)=\ell(v_5)=\ell(v_6)=g-h$.
For the pendant neighbors of $v_2$, label them such that their label sum is $-h$.
Then $M_9(p_1,p_2)$ is $\A$-vertex magic graph with magic constant $g$.
\end{proof}

\begin{proposition}\label{pro4.9}
Let $\A$ be a finite abelian group with $|\A|\geq 3$.
Then $M_{10}(p_1,p_2)$ is $\A$-vertex magic if and only if $p_1=p_2=0$ and the order of $\A$ is even.
\end{proposition}
\begin{proof}
Suppose $M_{10}(p_1,p_2)$ is $\A$-vertex magic with magic constant $g$ under mapping $\ell$.
Without loss of generality, if $p_1\neq0$, then $\ell(v_6)=\omega(v_5)-\ell(v_1)=g-g=0$, which is a contradiction.
With the similar argument, we obtain $p_2=0$.
For $M_{10}(0,0)$, observe that \begin{align*}
\omega(v_1)=\ell(v_2)+\ell(v_4)+\ell(v_5)+\ell(v_6)=\omega(v_3)+\ell(v_5)+\ell(v_6),
\end{align*}
so $\ell(v_5)=-\ell(v_6)$.
Since $\ell(v_1)=\omega(v_6)-\ell(v_5)=g-\ell(v_5)$ and $\ell(v_1)=\omega(v_5)-\ell(v_6)=g-\ell(v_6)=g+\ell(v_5)$, we obtain $2\ell(v_5)=0$.
Hence, $\ell(v_5)$ is an involution of $\A$ and then the order of $\A$ is even.

Conversely, let $h$ be an involution of $\A$.
For $M_{10}(0,0)$, define $\ell(v)=h$ for any $v\in V$, then $M_{10}(0,0)$ is a vertex magic graph with magic constant $0$.
\end{proof}

\begin{proposition}\label{pro4.10}
Let $\A$ be a abelian group.
Then $M_{12}(p_1,p_2)$ is $\A$-vertex magic if and only if $p_2=0$, $\A$ contains an involution $h$ and there exist $g_1,g_2\in \A \setminus\{0\}$ such that $g_1+g_2=h$.
\end{proposition}
\begin{proof}
Suppose $M_{12}(p_1,p_2)$ is $\A$-vertex magic with magic constant $g$ under mapping $\ell$.
If $p_2\neq0$, then $\ell(v_4)=\omega(v_5)-\ell(v_2)=0$, which is a contradiction.
Hence, $p_2=0$.
For $M_{12}(p_1,0)$,
\begin{align*}
\omega(v_4)=\ell(v_2)+\ell(v_3)+\ell(v_5)=\omega(v_5)-\ell(v_4)+\ell(v_3)+\ell(v_5),
\end{align*}
which gives that $\ell(v_3)+\ell(v_5)=\ell(v_4)$.
Thus,
\begin{align*}
\ell(v_3)+\ell(v_5)+\ell(v_4)=2\ell(v_4)=\omega(v_2)-\ell(v_1)=0,
\end{align*}
which implies  $\ell(v_4)$ is an involution of $\A$.

Conversely, let $h$ be an involution of $\A$ and there exist $g_1,g_2\in \A\setminus\{0\}$ such that $g_1+g_2=h$.
Let $g\in \A\setminus\{0,h\}$.
For $M_{12}(p_1,0)$, define $\ell(v_1)=g$, $\ell(v_3)=g_1$, $\ell(v_4)=h$, $\ell(v_2)=g+h$ and $\ell(v_5)=g_2$.
For the pendant neighbors of $v_1$, label them such that their label sum is $h$.
Then $M_{12}(p_1,0)$ is $\A$ vertex magic with constant $g$.
\end{proof}

\begin{proposition}\label{pro4.13}
Let $\A$ be a abelian group with $|\A|\geq3$.
Then  $M_{14}(p_1,p_2)$ is $\A$-vertex magic if and only if $p_1=p_2=0$ and there exist $g\in \A$ such that $g=2h_1=2h_2$ for different $h_1,h_2\in \A\setminus\{0\}$.
\end{proposition}
\begin{proof}
Suppose $M_{14}(p_1,p_2)$ is $\A$-vertex magic with magic constant $g$ under mapping $\ell$.
If $p_1\neq0$, then $\ell(v_1)=g$ and $\ell(v_5)=\omega(v_6)-\ell(v_1)=0$, which is a contradiction.
Similarly, it is easy to prove that $p_2=0$.
According to $\ell(v_5)=\omega(v_6)-\ell(v_1)=g-\ell(v_1)$ and  $\ell(v_2)=\ell(v_4)=g-\ell(v_6)$, we have \begin{align*}
\omega(v_4)=\ell(v_2)+\ell(v_3)+\ell(v_5)=\ell(v_2)+\omega(v_2)-\ell(v_1)-\ell(v_4)+\omega(v_6)-\ell(v_1)=2g-2\ell(v_1)=g,
\end{align*}
which gives that $2\ell(v_1)=g$.
Moreover, since $\omega(v_3)=\ell(v_2)+\ell(v_4)=2g-2\ell(v_6)=g$, then $2\ell(v_6)=g$ follows.
Observe that $\ell(v_3)=\omega(v_2)-\ell(v_1)-\ell(v_4)=\ell(v_6)-\ell(v_1)\neq0$, so $\ell(v_6)\neq\ell(v_1)$.

Conversely, $\A$ contains element $g$ such that $g=2h_1=2h_2$ for different $h_1,h_2\in \A\setminus\{0\}$.
For $M_{14}(0,0)$, define $\ell(v_1)=\ell(v_5)=h_1$, $\ell(v_2)=\ell(v_4)=h_2$,  $\ell(v_3)=h_2-h_1$ and $\ell(v_6)=h_2$.
Then $M_{14}(0,0)$ is $\A$-vertex magic with magic constant $g$ under mapping $\ell$.
\end{proof}

In particular, let $\A=Z_p$, where $p$ is an odd  prime.
Since the order of  each non-identity element of $Z_p$ is $p$, we obtain the next corollary.
\begin{corollary}\label{cor4.11}
Graphs $M_9(p_1,p_2)$, $M_{10}(p_1,p_2)$, $M_{12}(p_1,p_2)$  and $M_{14}(p_1,p_2)$ are not group vertex magic.
\end{corollary}

\begin{proposition}\label{pro4.12}
Let $\A$ be a abelian group.
Then $M_{11}(p_1,p_2)$ is $\A$-vertex magic if and only if $p_1=p_2=0$.
Further, $M_{11}(0,0)$ is group vertex magic.
\end{proposition}
\begin{proof}
Suppose $M_{11}(p_1,p_2)$ is $\A$-vertex magic with magic constant $g$ under mapping $\ell$.
If $p_1\neq0$, then $\ell(v_1)=g$ and  $\ell(v_6)=\omega(v_7)-\ell(v_1)=0$, which is a contradiction.
Hence, $p_1=0$.
Using the similar discussion, we have $p_2=0$.
For $M_{11}(0,0)$ and any abelian group $\A$, let $g\in \A\setminus \{0\}$ and define $\ell(v_1)=\ell(v_2)=\ell(v_5)=-g$ and $\ell(v_3)=\ell(v_4)=\ell(v_6)=\ell(v_7)=g$.
Then, $M_{11}(0,0)$ is $\A$-vertex magic with magic constant $0$.
\end{proof}

\begin{proposition}\label{pro4.14}
For any abelian group $\A$, $M_{13}(p_1,p_2)$ is not $\A$-vertex magic.
\end{proposition}
\begin{proof}
If $M_{13}(p_1,p_2)$ is $\A$-vertex magic graph with magic constant $g$ under $\ell$, for  some abelian group $\A$.
Since the diameter of $M_{14}(p_1,p_2)$ is $3$, $p_1$ and $p_2$ cannot all be $0$.
Without loss of generality, assume that $p_1\neq0$, then $\ell(v_1)=g$ and $\ell(v_4)=\omega(v_5)-\ell(v_1)=0$, which is a contradiction.
Hence, $M_{13}(p_1,p_2)$ is not $\A$-vertex magic for any abelian group.
\end{proof}

To summarize what we have proved, the theorem below is obtaind.
\begin{theorem}
Let $G$ be a bicyclic graph with diameter $3$.
Then $G$ is group vertex magic if and only of $G=M_{11}(0,0)$.
\end{theorem}

\end{sloppypar}

\end{document}